\documentclass[notitlepage]{amsart}

\usepackage{amsfonts}
\usepackage[pdftex]{graphicx}
\usepackage[latin1]{inputenc}
\usepackage[all]{xy} 
\usepackage{latexsym,amssymb,amscd}
\usepackage{amsmath}
\usepackage{color}
\usepackage{cite}
\usepackage{multicol}
\usepackage{pb-diagram, pb-xy}
\usepackage{graphicx}
\usepackage{tikz}
\usetikzlibrary{arrows}
 \usetikzlibrary{snakes}
 \usetikzlibrary{shapes}
\usetikzlibrary{calc}

\newtheorem{pro}{Proposition}[section]
\newtheorem{teo}[pro]{Theorem}
\newtheorem{defi}[pro]{Definition}
\newtheorem{lem}[pro]{Lemma}
\newtheorem{cor}[pro]{Corollary}
\newtheorem{rk}[pro]{Remark}

\newtheorem{ex}[pro]{Example}

\newcommand{\C}{\mathcal{C}}

\newcommand{\Q}{\mathbb{Q}}

\newcommand{\F}{\mathfrak{F}}

\newcommand{\Z}{\mathbb{Z}}

\newcommand{\cP}{\mathcal{P}}
\newcommand{\modu}{\mathrm{mod}}
\newcommand{\Modu}{\mathrm{Mod}}
\newcommand{\proj}{\mathrm{proj}}
\newcommand{\diag}{\mathrm{diag}}
\newcommand{\Mat}{\mathrm{Mat}}

\newcommand{\Ext}{\mathrm{Ext}}
\newcommand{\Hom}{\mathrm{Hom}}
\newcommand{\End}{\mathrm{End}}
\newcommand{\Ker}{\mathrm{Ker}}
\newcommand{\Coker}{\mathrm{CoKer}}
\newcommand{\Ima}{\mathrm{Im}}
\newcommand{\pd}{\mathrm{pd}}
\newcommand{\findim}{\mathrm{fin.dim}}

\newcommand{\Tr}{\mathrm{Tr}}
\newcommand{\rad}{\mathrm{rad}}

\newcommand{\add}{\mathrm{add}}

\newcommand{\ra}{\rightarrow}

\newcommand{\ssa}{$(\Lambda, \, \leq)\ $}

\newenvironment{dem}{\noindent\bf Proof. \rm }{$\ \Box$}
\usepackage{latexsym,amssymb,amscd}
\usepackage{amsmath}

\begin{document}
\title[Cokernels of the Cartan Matrix and Stratifying Systems]{Cokernels of the Cartan Matrix and Stratifying Systems}
\author{E.N.Marcos, O. Mendoza, C. S\'aenz.}
\thanks{
The authors thanks the Project PAPIIT-Universidad Nacional Aut\'onoma de M\'exico IN103317. The first named author  was partially supported by the Projeto tem\'atico FAPESP 2014/09310-5}
\footnote{We thank Maria Izabel Ramalho Martins (Bel), for her friendship and hospitality.}
\begin{abstract} We study the cokernel of the application given by the Cartan Matrix $C_\Lambda$ of a finite dimensional $k$-algebra $\Lambda.$ This produces a finitely generated abelian group, the Cartan group $G_\Lambda,$ which is invariant under derived equivalences. We are interested in the case when 
$G_\Lambda$ is finite. For a standardly stratified algebra, it is shown that this group is always finite and  some interesting connections with the standard modules are found. As a consequence, it is got that $G_\Lambda$ can be seen as a measure of how far is a standardly stratified algebra $\Lambda$ to be quasi-hereditary. 
Finally, it is also shown that any finite abelian group can be realized as the Cartan group of some standardly stratified algebra.
\end{abstract}
\maketitle
\section{Introduction.}
The Cartan matrix $C_\Lambda$ of a finite dimensional $k$-algebra $\Lambda$ has been an important tool to be used in different contexts appearing in representation theory of algebras, homological algebra and graphs. For example, it is well known that if $\Lambda$ is a  finite dimensional $k$-algebra of 
finite global dimension, then $\mathrm{det}\,C_\Lambda=\pm 1.$ There is a conjecture, which states that $\mathrm{det}\,C_\Lambda= 1$ for any 
finite dimensional $k$-algebra $\Lambda$ of finite global dimension. This homological conjecture, for the Cartan matrix, is still open and only partial results have been obtained so far, see for example in  \cite{W,Z}. There are other notions that are constructed by using the Cartan matrix $C_\Lambda,$ namely, the Euler characteristic and the Coxeter transformation of $\Lambda,$ which play an important role in the Auslander-Reiten theory \cite{ARS, ASS}. Thus, the idea of obtaining new concepts, by using the Cartan matrix, seems to be very fruitful in the search of useful tools.
\

In this paper, the term algebra means finite dimensional $k$-algebra and all the modules to be considered are finite dimensional left modules. Given an algebra $\Lambda,$ the aim of this paper is to assign, in a useful way, a finitely generated  abelian  group $G_\Lambda$ (the Cartan group) and  to study the properties that this group has by considering some classes of algebras.  The class we are interested in is the one given by the standardly stratified algebras, which were introduced by V. Dlab \cite{Dlab}. We point out, that this class of algebras has been extensively studied 
\cite{ADL, AHLU, Frisk1, Frisk2, FriskMaz, FKM, Mazor1, Mazor2,Mazor3,MazOv,MazPark,MSX,MenSan,Webb,Xi} and contains the very important subclass of quasi-hereditary algebras introduced by L. Scott in \cite{S}.
\

Let $\Lambda$ be a $k$-algebra. We denote by $\modu\,(\Lambda)$ the category of finitely generated $\Lambda$-modules and $\proj\,(\Lambda)$ the class of finitely generated projective $\Lambda$-modules. For any class $\C$ in $\modu\,(\Lambda),$ which is closed under extensions and direct summands, the Grothendieck group of $\C$ is denoted by $K_0(\C).$ The Cartan map is the $\Z$-linear transformation 
 $$C=C_{\Lambda}:K_0(\proj\,(\Lambda))\to K_0(\Lambda),\quad [P_i]\mapsto\dim(P_i),$$ 
 where $\{P_i\}_{i=1}^n$ is a complete set of pairwise non isomorphic indecomposable objects in $\proj\,(\Lambda)$ and 
 $K_0(\Lambda):=K_0(\modu\,(\Lambda)).$ We define the Cartan group $G_\Lambda$ of $\Lambda$ as the cokernel of the  Cartan map 
 $C_\Lambda.$ Note that $G_\Lambda$ is a finitely generated abelian group, since the abelian groups $K_0(\proj\,(\Lambda))$ and 
 $K_0(\Lambda)$ are both isomorphic to $\mathbb{Z}^n.$ We identify any matrix 
$C\in\Mat_{n\times n}(\Z),$ with the $\Z$-linear transformation $\Z^n\to\Z^n,$ $X\mapsto CX.$ In section 2, for the sake of completeness, we review some facts and notions related to
Grothendieck groups, standardly stratified systems and  related facts.
\

In section 3, we recall some well known results and notations concerning with the theory of finitely generated abelian groups, which  are fundamental 
for the development of the paper. For $D\in\Mat_{n\times n}(\Z),$ it is shown that the abelian group $G_D:=\Coker\,(D)$ is finite if and only if 
$\det\,(D)\neq 0.$ Moreover, if $G_D$ is finite, then $|G_D|=|\det\,(D)|.$
\

Let $(\Theta,\underline{Q},\leq)$ be an Ext-projective stratifying system  of size $t$ in $\modu\,(\Lambda).$ and let $B=\End_\Lambda(Q)^{op}.$  
By \cite[Lemma 2.1]{MMSZ}, it is known that the Grothendieck group $K_0(\F(\Theta))$ is free of rank $t,$ with a basis formed by  each  image 
$[\Theta(i)]$ of $\Theta(i)$ under the canonical map $F(\F(\Theta))\to  K_0(\F(\Theta)),$ where  $F(\F(\Theta))$ is the free abelian group on the $\Theta$-filtered $\Lambda$-modules,  we give more details  in sections 2 and  4. Consider the $\Theta$-Cartan matrix $C_\Theta\in\Mat_{t\times t}(\Z),$ where $[C_\Theta]_{ij}:=\dim_k\Hom_\Lambda(Q(i),\Theta(j)).$ One of the main result, in section 4, can be summarized as follows (see Theorem \ref{Mepss}).
\

{\bf Theorem A} Let $(\Theta,\underline{Q},\leq)$ be an Ext-projective stratifying system of size $t$ in $\modu\,(\Lambda),$ and let $B:=\End_\Lambda(Q)^{op}.$ Then, the following statements hold true.
 \begin{itemize}
 \item[(a)] $G_B\simeq \Coker\,(C_{\Theta})$ and $|G_B|=\prod_{i=1}^t\,\dim_k\End_\Lambda(\Theta(i)).$
 \item[(b)] The exponent of $G_B$ is a multiple of $\dim_k\End_\Lambda(\Theta(i)),$ for any $i\in[1,t].$
 \end{itemize}
 
 There are several consequences that can be obtained from the result above. We enumerate and summarize them in the following way. For a complete 
 description and proof, the reader can see in Corollary \ref{DiagD}, Corollary \ref{A1}, Remark \ref{DiagD} and Corollary \ref{A2}.
 \
 
 {\bf Corollary B} Let $(\Lambda,\leq)$ be a standardly stratified $k$-algebra. Then, the following statements hold true.
 \begin{itemize}
 \item[(a)] $G_\Lambda\simeq \Coker\,(C_\Delta)$ and $|G_\Lambda|=\prod_{i=1}^n\,\dim_k\End_\Lambda(\Delta(i)).$
 \item[(b)] If $(\Lambda,\leq)$ is weakly triangular, then $C_\Delta=\diag(d_1,\cdots, d_n)$ and $G_\Lambda = \bigoplus^n_{i=1}\frac{\Z}{d_i\Z},$ where $d_i:=\dim_k\,(\Delta(i)).$ 
 \item[(c)] $(\Lambda, \leq)$ is quasi-hereditary if and only if $G_\Lambda=0.$
 \end{itemize}

It is quite important the case, see in section 4, when the $\Delta$-Cartan matrix $C_{{}_\Lambda\Delta}$ of a $k$-algebra $\Lambda$  is a diagonal matrix. In this situation, we obtain an explicit description of the Cartan group $G_\Lambda.$ The following result gives necessary and sufficient conditions, for a standardly stratified algebra 
$\Lambda,$ to have that the  $\Delta$-Cartan matrix be a diagonal one. A full statement and its proof is the Theorem \ref{EqDiagM}.
\

{\bf Theorem C} For a quotient path $k$-algebra  $\Lambda=kQ/I,$ where $I$ is an admissible ideal,  the following statements are equivalent.
 \begin{itemize}
  \item[(a)] $C_{{}_\Lambda\Delta}$ is a diagonal matrix and $\Lambda$ is a standardly stratified algebra.
  \item[(b)]  $\Lambda$ is isomorphic to a triangular matrix $k$-algebra  
  $\begin{pmatrix}
   R & 0\\
        M      & T
 \end{pmatrix},$ where $M$ is a $T-R$-bimodule. Moreover $R$ and $T$ are standardly stratified 
  $k$-algebras satisfying that their corresponding $\Delta$-Cartan matrices $C_{{}_R\Delta}$ and $C_{{}_T\Delta}$ are diagonal and $M\in\F({}_T\Delta).$
 \end{itemize}
 
 Section 5 is devoted to the study of the Cartan group $G_\Lambda$ and the $\Delta$-Cartan matrix $C_\Delta$ for  radical square zero path $k$-algebras $\Lambda.$  Such an algebra $\Lambda$  is of the form $\Lambda=kQ/J^2,$ where $J$ is the ideal 
of $kQ$ which is generated by the set of arrows $Q_1$ of the quiver $Q.$ Let $v\in Q_0$ be a vertex of $Q.$ The idempotent in $\Lambda,$ attached to the vertex $v,$ is denoted by 
$e_v.$ Then, associated with the vertex $v,$ we have: the projective $\Lambda$-module $P(v):=\Lambda e_v$ and the simple $\Lambda$-module 
$S(v):=P(v)/\rad\,P(v).$ The following result gives us the order of the group $G_\Lambda,$ see Theorem \ref{RS04} for a complete version. 
\

{\bf Theorem D} Let $\Lambda=kQ/J^2$ and $\leq$ be a linear order on $Q_0$ such that $(\Lambda,\leq)$ is a standardly stratified $k$-algebra. If $Q$ does not have  sinks then $|G_\Lambda|=\prod_{v\in Q_0}\,[P(v):S(v)].$
\

In order to state the last main result of this section, we need to recall the following notions. Let $\Lambda=kQ/I,$ where $I$ is an admissible ideal of 
$kQ.$ It can be defined a pre-order relation $\leq_Q,$ on the set of vertices $Q_0,$ as 
 follows. For $v,w$ in $Q_0,$ we set $v\leq_Q w$ if there is an oriented path $\gamma$ in $Q$ starting at $v$ and ending at $w.$ Let 
 $\leq$ be a pre-order relation on $Q_0.$ We say that $\leq$ is a {\bf refinement} of $\leq_Q$ if the relation $x\leq_Q y$ implies that $x\leq y.$ For 
 each vertex $v\in Q_0,$ we denote by $\mathrm{loop}(v)$ the number of loops in $Q$  starting at the vertex $v.$ It is said that  $v\in Q_0$ is a {\bf quasi-source} if there is not an arrow $w\to v$ in $Q_1,$ with $w\neq v.$ 
\
 
 The last result in section 5 is the
 Theorem \ref{TR3} and can be written as follows.
\

{\bf Theorem E} For $\Lambda=k Q/J^2,$ the following statements are equivalent.
\begin{itemize}
\item[(a)] $(\Lambda,\leq)$ is a standardly stratified $k$-algebra for some linear refinement $\leq$ of $\leq_Q.$
\item[(b)] $Q$ does not have proper oriented cycles and the possible loops in $Q$ appear only in quasi-sources.
\item[(c)] $(\Lambda,\leq)$ is a standardly stratified $k$-algebra for any linear refinement $\leq$ of $\leq_Q.$
\end{itemize}
Moreover, if one of the above equivalent conditions holds true, then the $\Delta$-Cartan matrix  $C_\Delta$ is diagonal. Moreover, for any $i\in Q_0$
 $$d_i:=[C_\Delta]_{ii}=1+\mathrm{loop}(i)=\dim_k\End_\Lambda(P(i))=\dim_k\End_\Lambda(\Delta(i)),$$
 and thus $G_\Lambda\simeq\bigoplus_{i=1}^n\,\Z/d_i\Z.$
\

Finally, throughout the paper there are several examples illustrating the results we have obtained so far.

\section{Preliminaries.}

Throughout this paper, the term $k$-algebra means finite dimensional $k$-algebra, for a fixed field $k.$ We recall that a $k$-algebra $\Lambda$ is {\bf elementary} if $\Lambda/\rad\,(\Lambda)\simeq k^n$ as $k$-algebras, for some natural number $n.$  In the case that the $k$-algebra $\Lambda$ be 
elementary, we have that  $\Lambda$ is isomorphic to the a quotient 
path $k$-algebra $kQ/I,$ for some admissible ideal $I$ \cite[III.1 Theorem 1.9]{ARS}.   The category of finitely generated left $\Lambda$-modules is 
denoted by $\modu\,(\Lambda).$ Unless otherwise specified, we will work  with finitely generated $\Lambda$-modules,
full subcategories and non-empty classes. We denote by $\proj\,(\Lambda)$ the full subcategory of 
$\modu\,(\Lambda),$ whose objects are the projective $\Lambda$-modules. 
\

{\sc Standardly stratified $k$-algebras.} Let $\Lambda$  be a $k$-algebra.  We fix a complete list $S_1, S_2,\cdots ,S_n,$   of pairwise 
non-isomorphic simple $\Lambda$-modules, and  a linear order $\leq$ on the set  $[1,n]:=\{1,2,\cdots,n\}.$ For each $i\in[1,n],$  $P_i$ is 
the projective cover of $S_i,$ and the $i$-th  standard module 
 is the quotient $\Lambda$-module $\Delta(i):=P_i/\Tr_{\oplus_{j>i} P_j}(P_i),$ where $\Tr_{\oplus_{j>i} P_j}(P_i)$ is the trace of $\oplus_{j>i} P_j$ in $P_i.$ It is well known that $\Delta(i)$ is the maximal quotient of $P_i$ with composition factors among 
the simple $\Lambda$-modules $S_j$ with $j\leq i.$ The set of the {\bf standard} $\Lambda$-modules is $\Delta=\{\Delta(1),\Delta(2),\cdots,\Delta(n)\}$ and depends on the given order $\leq$ on the set  $[1,n].$ 
\

For a given class $\Theta$ of $\Lambda$-modules,  we denote by $\F(\Theta)$ the full subcategory of $\modu\,(\Lambda)$ whose objects are  the
$\Lambda$-modules $M$ which have a $\Theta$-filtration. That is, $M\in\F(\Theta)$ is  there is a finite 
chain $$0=M_0\leq M_1\leq M_2\leq\cdots \leq M_t=M$$

\noindent of submodules of $M$ such that each quotient $M_i/M_{i-1}$ is isomorphic to a module in $\Theta.$ In case that the class 
$\Theta:=\Delta,$ the modules in $\F(\Delta)$ are called $\Delta-$good modules.  
\

 If ${}_\Lambda\Lambda\in\F(\Delta)$ then the pair $(\Lambda,\leq)$ is said to be a (left) {\bf standardly stratified} algebra.
A standardly stratified algebra is called quasi-hereditary if the endomorphism ring of each standard module is a division ring. Quasi-hereditary 
algebras were introduced in \cite{CPS, S} to deal with highest weight categories, which play an important role in the representation theory of Lie algebras and algebraic groups.
\

{\sc Stratifying systems.} Let $\Lambda$ be a $k$-algebra. The concept of Ext-projective stratifying system (epss, for short) was introduced  in \cite{ES, MMS2}. This is a generalization of the standard modules and it is very useful for constructing 
standardly stratified algebras. 
 
\begin{defi}\cite{MMS2}\label{epss}
Let $\Theta=\{\Theta(i)\}_{i=1}^t$ be a family of non-zero $\Lambda$-modules,
$\underline{Q}=\{Q(i)\}_{i=1}^t$ a family of indecomposable $\Lambda$-modules
and $\leq$  a linear order on $[1,t].$
The triple $(\Theta,\underline{Q}, \leq )$ is an Ext-projective stratifying system of size $t,$ in $\modu\,(\Lambda),$
if the following three conditions hold true
\begin{enumerate}
\item[(a)] $\mathrm{Hom}_\Lambda(\Theta(j),\Theta(i))=0$ for $j> i,$
\item[(b)] for each $i\in [1,t]$, there is an exact sequence in $\modu\,(\Lambda)$ 
$$0\to K(i)\rightarrow Q(i)\stackrel{\beta_i}\rightarrow\Theta(i) \to 0$$
such that  $K(i)\in\mathcal{F}(\{\Theta(j): j>i\}),$
\item[(c)] $\Ext^1_\Lambda(Q,-)|_{\F(\Theta)}=0$ for $Q:=\oplus_{i=1}^tQ(i).$
\end{enumerate}
\end{defi}

 It is shown in \cite{MMS2}  that $B:=\mathrm{End}({}_{R}Q)^{op}$ is  a standardly stratified algebra with respect to the given linear order $\leq$ on the set $[1,t].$ In this case, the standard $B$-modules are computed by using the family $\{{}_BP(i)\}_{i\in[1,t]}$ of projective $B$-modules, where 
 ${}_BP(i):=\Hom_\Lambda(Q,Q(i)).$ 
\

For any $k$-algebra $\Lambda,$ there exists  Ext-projective stratifying system of size $n:=rk\,K_0(\Lambda).$ The canonical one is of the  
form $(\Delta, \underline{Q},\leq).$ Moreover, the pair $(\Lambda,\leq)$ is a standardly stratified algebra if, and only  if, 
$\underline{Q}=\{P_i\}_{i=1}^n$ \cite{MMS2}. Note that in this case, the algebra $B:=\mathrm{End}({}_{R}Q)^{op}$ is morita equivalent to 
$\Lambda.$

{\sc The Cartan group.}  Let $\Lambda$ be a $k$-algebra. Consider a class  $\C\subseteq\modu\,(\Lambda)$ of $\Lambda$-modules, which is 
closed under extensions and direct summands. Let $F(\C)$ be the free abelian group on the objects of $\C,$ and let $R(\C)$ be the subgroup of 
$F(\C)$ generated by the elements $A+B-C$ if there is an exact sequence 
$0\to A\to C\to B\to 0.$ The Grothendieck group of $\C$ is the quotient $K_0(\C):=F(\C)/R(\C).$ 
\

The Grothendieck group of $\modu\, (\Lambda)$ will be denoted by $K_0(\Lambda)$. Recall that this group is free abelian  and its canonical basis is 
the set of iso-classes of simple $\Lambda$-modules $S_1,S_2,\cdots,S_n.$ In particular, it is isomorphic to $\Z^n.$ On the other hand, the Grothendieck group $K_0(\proj\,(\Lambda))$ is also isomorphic to the free abelian group $\Z^n, $ and its  canonical basis is given by the set of iso-classes of indecomposable projective 
$\Lambda$-modules $P_1,P_2,\cdots,P_n.$
\

 As usual,  for $M\in\modu\, (\Lambda),$ we also denote by $\dim(M)$ its class $[M]$ in $K_0(\Lambda).$ It can be seen that 
 $\dim(M)=\sum_{i=1}^n\,[M:S_i][S_i],$ where $[M:S_i]$ is the multiplicity of $S_i$ in $M.$  The {\bf Cartan map} is the $\Z$-linear transformation 
 $$C=C_{\Lambda}:K_0(\proj\,(\Lambda))\to K_0(\Lambda),\quad [P_i]\mapsto\dim(P_i).$$
The {\bf Cartan group} $G=G_\Lambda$ of $\Lambda$ is the cokernel of the  Cartan map $C_\Lambda.$ Note that $C_\Lambda^t=C_{\Lambda^{op}},$ since 
the functor $\Hom_\Lambda(-,\Lambda):\proj(\Lambda)\to\proj\,(\Lambda^{op})$ is a duality.
\

It is clear that, for an  algebra $\Lambda$  of finite global dimension, its Cartan group $G_\Lambda$ is zero, since $C_\Lambda$ is invertible. 
 \
 
 We state next, the following proposition of \cite{BHL}. We recall that two $k$-algebras $A$ and $B$ are derived equivalent if their corresponding 
 bounded derived categories $D^b(\modu\,(A))$ and $D^b(\modu\,(B))$ are equivalent as triangulated categories.

\begin{pro} \cite[Proposition 2.1]{BHL} \label{InvDer} Let $A$ and $B$ be  derived equivalent $k$-algebras. Then, there exists some integral invertible matrix $P$ such that $PC_AP^t=C_B.$
\end{pro}

\begin{rk} Let $A$ and $B$ be finite dimensional $k$-algebras, which are derived equivalent. Then, by Proposition \ref{InvDer} we get that $G_A\simeq G_B.$
\end{rk}

 We are interested in the case that the Cartan map $C_\Lambda$ has a non vanishing determinant, and hence by Lemma \ref{Eq} it is equivalent to say that the Cartan group 
 $G_\Lambda$ is finite. In order to prove that, we use  various elementary results of presentations of finite abelian groups. In section 3 we deal with 
this and also some proofs are given for the sake of completeness.
\

{\sc Weakly-triangular $k$-algebras.}
Let $\Lambda=kQ/I$ be a quotient path $k$-algebra, where $I$ is an admissible ideal. Let $\leq$ be  a linear order on the set of vertices
$Q_0:=\{v_1, v_2,\ldots,v_n\}.$  We recall that the algebra $\Lambda$ is {\bf weakly triangular}, with respect to
the partially ordered set $(Q_0,\leq),$  if $\Hom_\Lambda(P(v_i),P(v_j))$ is equal to zero for $v_i<v_j,$ where $P({v_t}):=\Lambda e_{v_t}$ and $e_{v_t}$ is
the idempotent associated with the vertex $v_t$, for each $t\in[1,n].$ In this case, we also say that the pair $(\Lambda,\leq)$ is a weakly triangular algebra. 
If in addition, $Q$ does not have oriented cycles, it is said that $(\Lambda,\leq)$ is a {\bf triangular} algebra. Note that a weakly triangular algebra 
$(\Lambda,\leq)$ does not have to be standardly stratified, however $(\Lambda,\leq^{op})$ is   standardly stratified.  
\

As a consequence of the following proposition, we get that the only possible oriented cycles in a weakly triangular algebra are the concatenations 
of loops. In order to state and prove this result, we recall the following notions. An {\bf oriented cycle} in $Q$ is {\bf proper} if it does not have loops, and a vertex $v\in Q_0$ is a {\bf quasi-source} if there is not an arrow $w\to v$ in $Q_1,$ with $w\neq v.$ 

\begin{pro}\label{WT} Let $\Lambda=kQ/I$ be a quotient path $k$-algebra, where $I$ is an admissible ideal. Then, $Q$ does not have proper oriented 
cycles if, and only if, there is a linear order $\leq$ on $Q_0$ such that $(\Lambda,\leq)$ is a weakly triangular algebra. 
\end{pro}
\begin{dem} $(\Rightarrow)$ Assume that $Q$ does not have proper oriented cycles. We prove firstly that $Q$ has quasi-sources. Indeed, suppose that 
$Q$ does not have quasi-sources. Fix some $v_1\in Q_0.$ Then there is an arrow $v_2\to v_1,$ with $v_2\neq v_1.$ Since, it is supposed that 
$Q$ does not have quasi-sources, we can repeat this procedure and thus we get an oriented path $v_n\to v_{n-1}\to \cdots\to v_2\to v_1,$ without loops, where $n=|Q_0|.$ Since $v_n$ is not a quasi-source, there is an arrow $v_j\to v_n$ with $v_j\neq v_n.$ Hence we get a proper oriented 
cycle $v_j\to v_{n}\to v_{n-1}\to \cdots\to v_j,$ which is a contradiction proving that $Q$ has quasi-sources.
\

In order to construct a linear order $\leq$ on $Q_0$ in such a way that $(\Lambda,\leq)$ be weakly triangular, we proceed as follows. Firstly, we 
choose a quasi-source $v_1\in Q_0.$ Then, for any vertex $j\in Q_0-\{v_1\}$ there are not oriented paths from $j$ to $v_1$ and thus 
$$\Hom_\Lambda(P(v_1), P(j))\simeq e_{v_1}\Lambda e_j=0\quad\text{for all }\, j\in Q_0-\{v_1\}.$$ In particular, $\Lambda$ is isomorphic to a matrix triangular $k$-algebra 
$\begin{pmatrix} \Sigma & 0\\ T & \Gamma \end{pmatrix},$
where $\Sigma:=\End_\Lambda(P(v_1))^{op},$ $\Gamma:=\End_\Lambda(Q(v_1))^{op},$ $Q(v_1):=\bigoplus_{j\in Q_0-\{v_1\}}\,P(j)$ and 
$T:=\Hom_\Lambda(Q(v_1),P(v_1)).$ Note that $\Gamma$ satisfies the same conditions as $\Lambda$ does, $Q_\Gamma=Q_\Lambda-\{v_1\}$ 
and $\Hom_\Lambda(\Lambda e_i,\Lambda e_j)\simeq \Hom_\Gamma(\Gamma e_i,\Gamma e_j)$ for any vertices $i,j\in Q_\Gamma.$ We choose now 
a quasi-source $v_2$ in $Q_\Gamma.$ Then, as before $\Hom_\Lambda(P(v_2),P(j))\simeq \Hom_\Gamma(P(v_2),P(j))=0$ for any 
$j\in Q_\Gamma-\{v_2\}.$ Therefore we set $v_1<v_2,$ and by iterating this procedure, we get a linear order $v_1<v_2<\cdots<v_n$ on $Q_0$ 
satisfying that $\Hom_\Lambda(P(v_i),P(v_j))=0$  for $v_i<v_j.$
\

$(\Leftarrow)$ Let $(\Lambda,\leq)$ be weakly triangular algebra. Suppose there is a proper oriented cycle $C.$ Without loss of generality, we may assume that $C$ is minimal, that is $C=v_1\to v_2\to \cdots\to v_n\to v_1$ and all the vertices in this cycle are different to each other.  We consider the numeration  of the vertices $v_j$ in $C,$ with $j\in\Z/n\Z,$ and thus $v_{n+1}=v_1.$ We assert there exist some $j_0\in\Z/n\Z$ such that 
$v_{j_0+1}<v_{j_0}.$ If this were not the case, we would have $v_1<v_2<\cdots v_n<v_{n+1}=v_1,$ which is a contradiction that proves our assertion. It can be assumed that $j_0=1.$ Then there is an arrow $v_1\to v_2,$ and thus $\Hom_\Lambda(P(v_2),P(v_1))\neq 0,$ with $v_2<v_1.$ This is a contradiction with the fact that 
$(\Lambda,\leq)$ is weakly triangular. Therefore there is not a proper oriented cycle in $Q.$
\end{dem}

 \section{Finite Abelian Groups}
 In this section, we introduce some well known notations and results that are fundamental for the development of the paper. 
 
 \begin{defi}
  A  {\bf quasi-elementary sequence}, of length $m\geq 1,$ is a sequence of non-negative integers $(f_1, \cdots, f_m)$  such that 
  $f_i \mid f_{i+1},$ for all 
  $i\in[1, m-1].$ If in addition, all the $f_i$ are positive, it is said that this sequence is  {\bf elementary}.
 \end{defi}

One of the main properties of integral square matrices can be summarized in the following well known result. 

\begin{pro}\label{V} Every matrix $C\in\Mat_{n\times n}(\Z)$ can be reduced, by integral elementary row and column transformations, to the form 
$\diag\,(f_1,f_2,\cdots,f_n),$ where $(f_1,f_2,\cdots,f_n)$ is a quasi-elementary sequence. Furthermore, the integers $f_1,f_2,\cdots,f_n$ are uniquely determined by the matrix $C.$
\end{pro}
\begin{dem} \cite[Proposition 9.13, Remark 9.16]{V}.
\end{dem}
\vspace{0.2cm}

The following Lemma will be very useful in all that follows. We identify any matrix 
$C\in\Mat_{n\times n}(\Z),$ with the $\Z$-linear transformation $\Z^n\to\Z^n,$ $X\mapsto CX.$

\begin{lem}\label{Eq} For any $D\in \Mat_{n\times n}(\Z)$ and $G:=\Coker\,(D),$  the following statements are equivalent.
\begin{itemize}
\item[(a)] $\det\,(D)\neq 0.$
\item[(b)] $D$ is a monomorphism as $\Z$-linear transformation.
\item[(c)] The abelian group $G$ is finite.
\end{itemize}
\end{lem}
\begin{dem} By tensoring  $D:\Z^n\to\Z^n$ with $\Q,$ we get  the $\Q$-lineal transformation $D\otimes_{\Z}\Q:\Q^n\to\Q^n.$ Note that $\det\,(D\otimes_{\Z}\Q)=\det\,(D).$ Thus, 
$\det\,(D)\neq 0$ iff $D\otimes_{\Z}\Q$ is an isomorphism. Furthermore, the functor $-\otimes_\Z\Q:\Modu\,(\Z)\to\Modu\,(\Q)$ is exact, since
 $\Q$ is a flat $\Z$-module.
 \
 
 (a) $\Rightarrow$ (b) Let $\det\,(D)\neq 0.$ Then, $D\otimes_{\Z}\Q$ is an isomorphism. Suppose that $0\neq \Ker\,(D).$ Then, 
 $\Ima\,(D)\simeq\Z^t,$ with $t<n,$ since $\Z^n=\Ker\,(D)\oplus \Ima\,(D).$ On the other hand, 
 $\Ima\,(D\otimes_{\Z}\Q)=\Ima\,(D)\otimes_{\Z}\Q\simeq \Q^t,$ contradicting that $D\otimes_{\Z}\Q$ is an isomorphism.
 \
 
 (b) $\Rightarrow$ (c) By hypothesis, we have an exact sequence $0\to \Z^n\xrightarrow{D}\Z^n\to G\to 0.$ In particular, $G$ is a finitely generated 
 abelian group. Therefore $G=\Z^t\oplus T,$ where $t$ is a non-negative integer and $T$ is a finite group. On the other hand, $D\otimes_{\Z}\Q$ is 
 a monomorphism since $\Q$ is flat, and thus, $D\otimes_{\Z}\Q$ is an isomorphism. Then, by tensoring the above exact sequence, it follows that 
 $0=G\otimes_{\Z}\Q=\Q^t\oplus (T\otimes_{\Z}\Q)=C;$ proving that $t=0$ and thus $G=T$ is a finite group.
\
 
 (c) $\Rightarrow$ (a) By tensoring with $\Q$ the exact sequence $\Z^n\xrightarrow{D}\Z^n\to G\to 0,$ we get the exact sequence 
 $\Z^n\otimes_{\Z}\Q\xrightarrow{D\otimes_{\Z}\Q}\Z^n\otimes_{\Z}\Q\to G\otimes_{\Z}\Q\to 0.$ Note that $G\otimes_{\Z}\Q=0,$ since $G$ is 
 finite. Therefore $D\otimes_{\Z}\Q$ is an isomorphism and thus $\det\,(D)=\det\,(D\otimes_{\Z}\Q)\neq 0.$
\end{dem}

\begin{pro}{\label{detord}} For any $D\in \Mat_{n\times n}(\Z),$ with $\det\,(D)\neq 0,$
and $G:=\Coker\,(D),$ we have that   $|G|= |\det\,( D)|.$
 \end{pro}
 \begin{dem} By Lemma \ref{Eq}, we know that $G$ is a finite group.  
 By Proposition \ref{V}, there is a positive integer sequence $(f_1,f_2,\cdots,f_n)$  and an integer matrix $C$  such that 
 $\det\,(C)=\pm 1$ and
 $D=C\,\diag\,(f_1,f_2,\cdots,f_n).$  
  Therefore $$G\simeq \Coker\,(\diag\,(f_1,f_2,\cdots,f_n))=\oplus_{i=1}^n\,\Z/f_i\Z$$  and thus 
 $|G|=\prod_{i=1}^n\,f_i=\det\,(\diag\,(f_1,f_2,\cdots,f_n))=|\det\,(D)|,$
 proving the result.
 \end{dem}

Let $G$ be a finite abelian group.  It is well known that there is a unique elementary sequence 
 $(f_1, \ldots, f_m)$ such that $f_i\geq 2,$ for all $i\in[2,m],$  and 
 $G\simeq \bigoplus_{i=1}^m \frac{\Z}{f_i\Z}.$ This elementary sequence  is called the {\bf invariant factors sequence} of the group $G,$ and the
 $f_i's$ are called the invariant factors  of $G$. Moreover, for any epimorphism
 $p:\Z^n \to G,$ it follows that $n\geq m.$
\

 We also recall that the exponent $\mathrm{exp}\,(G),$ of  $G,$ is the minimal natural number $t$ such that $tg= 0$ for every 
 $g\in G$. 
 
 \begin{pro}{\label{ddividee}}
  Let $D\in \Mat_{n\times n}(\Z)$ be a triangular  matrix, with non-zero determinant, and let $t=\mathrm{exp}\,(G)$ be the
  exponent of $G:=\Coker(D).$ Then $t$ is a multiple of $d_i:=[D]_{i,i},$ for every $i\in[1,n].$
 \end{pro}
 \begin{proof} Assume that $D$ is upper triangular. Let $d_{i,j}:=[D]_{i,j}.$ Without loss of generality, we may assume that all the $d_i$ are positive. For any $j\in[1,n],$ we  consider the matrix $E_j \in\Mat_{n\times 1}(\Z),$ where $[E_j]_{i,1}:=\delta_{i,j}.$
 \
 
  For $D:\Z^n\to \Z^n,$ we have that the image of $D$ is generated by the column vectors $E_1,E_2,\cdots, E_n.$ Furthermore, using the fact that 
  $D$ is upper triangular, we get for any $i\in[1,n]$ that 
  $$(*)\quad \left\langle[D]^1,[D]^2,\cdots,[D]^i\right\rangle\subseteq \left\langle E_1,E_2,\cdots, E_i\right\rangle.$$
  Let $i\in[1,n].$ Since $tG=0,$ we have that  $t\Z^n \subseteq \Ima\,(D)$ and thus
  \begin{equation}\label{colunas}
  tE_i= \sum_{j=1}^n \lambda_{i,j}[D]^j
  \end{equation}
  We assert that $\lambda_{i,j}=0,$ for any $j\in[i+1,n].$ Indeed, suppose $\exists\,j\in[i+1,n]$ such that $\lambda_{i,j}\neq0.$ Without loss of generality, we may assume that $\lambda_{i,r}=0,$ for any $r\in[j+1,n].$ Let $u:=\sum_{r=1}^{j-1} \lambda_{r,i}[D]^r.$ Then, by $(*)$ it 
  follows that $u\in\left\langle E_1,E_2,\cdots, E_{j-1}\right\rangle.$ Moreover, from (1) and $(*)$, we obtain 
  $$\lambda_{i,j}d_jE_j=tE_i-u-\lambda_{i,j}\sum_{r=1}^{j-1}\,d_{r,j}E_r\in\left\langle E_1,E_2,\cdots, E_{j-1}\right\rangle ,$$
  which is a contradiction, since $\lambda_{i,j}d_j\neq 0.$ Thus, our assertion follows and then (1) can be written as  
  $tE_i= \sum_{j=1}^i \lambda_{i,j}[D]^j.$ Finally, from the preceding equality and the fact that $D$ is upper triangular, we conclude that 
  $t=\lambda_{i,i}d_i,$ proving the result.
  \end{proof}
  
  \begin{ex}\label{Ej1} Consider the quotient path $k$-algebra $\Lambda=k\mathcal{Q}/I,$ where $\mathcal{Q}$ is the quiver 
\xymatrix{ 1\ar@<1ex>[r]^\alpha & 2  \ar@<1ex>[l]^\beta  }
   and $I=\left\langle \beta\alpha\beta\right\rangle.$ The structure of the indecomposable projective $\Lambda$-modules is as follows
   $$P(1)=\begin{matrix} 1\\2\\ 1\\ 2\end{matrix}\quad\text{and}\quad P(2)=\begin{matrix} 2\\ 1\\ 2.\end{matrix}$$
In this case, the Cartan matrix is $C_\Lambda=\begin{pmatrix} 2 & 1\\ 2 & 2\end{pmatrix}.$ By doing integral elementary row and column transformations, we get 
$$\begin{pmatrix} 2 & 1\\ 2 & 2\end{pmatrix}\sim\begin{pmatrix} 2 & 1\\ 0 & 1\end{pmatrix}\sim\begin{pmatrix} 2 & 0\\ 0 & 1\end{pmatrix}
\sim\begin{pmatrix} 1 & 0\\ 0 & 2\end{pmatrix}.$$
Therefore, the Cartan group $G_\Lambda$ of $\Lambda$ is $\Z/2\Z.$
\end{ex}
 
 \section{Cartan groups and standardly stratified algebras} 
 
 We are interested in the study of the Cartan groups for the class of standardly stratified $k$-algebras. For a $k$-algebra $\Lambda,$  we fix 
 $n:=rk\,K_0(\Lambda).$ In this section, we study more closely the relationship between $G_\Lambda$ and $\Lambda.$
 \
 
 Let $(\Theta,\underline{Q},\leq)$ be an Ext-projective stratifying system,  of size $t,$ in $\modu\,(\Lambda).$ By \cite[Lemma 2.1]{MMSZ}, we 
 know that the Grothendieck group $K_0(\F(\Theta))$ is free of rank $t,$ with a basis formed by  each  image $[\Theta(i)]$ of $\Theta(i)$ under 
 the canonical map
 $F(\F(\Theta))\to  K_0(\F(\Theta)).$ Following \cite{MMS2}, we have the standardly stratified algebra $(B,\leq),$ where $B:=\End_\Lambda(Q)^{op}$ 
 and $Q:=\oplus_{i=1}^t\,Q(i).$ The family of standard $B$-modules ${}_B\Delta$ is computed by using the complete family ${}_B P:=\{{}_B P(i)\}_{i=1}^t$ of pairwise non-isomorphic  indecomposable projective $B$-modules, where ${}_B P(i):=\Hom_\Lambda(Q,Q(i))).$ Let  ${}_B S(i):=\mathrm{top}\,({}_B P(i)),$ for any 
 $i\in[1,t].$ We have the following diagram in the category 
 of abelian groups
 $$\xymatrix{      K_0(\add\,(Q))\ar[d]_{C_{Q,\Theta}} \ar[r]^{C_{Q}} & K_0(B)\\
 K_0(\F(\Theta)), \ar[ru]_{C_{\Theta}} & &
 } $$
 where the morphisms above are defined, on the canonical basis, as follows
 \begin{align*}
 C_{Q}[Q(j)]:=\sum_{i=1}^t\,\dim_k\Hom_\Lambda(Q(i),Q(j))[{}_B S(i)],\\
 C_{Q,\Theta}[Q(j)]:=\sum_{i=1}^t\,[Q(j):\Theta(i)][\Theta(i)],\\
 C_{\Theta}[\Theta(j)]:=\sum_{i=1}^t\,\dim_k\Hom_\Lambda(Q(i),\Theta(j))[{}_B S(i)].
 \end{align*}
 
 In the following Lemma, we summarize the main properties of the above morphisms.
 
 \begin{lem}\label{M} For any epss $(\Theta,\underline{Q},\leq)$ of size $t,$ in $\modu\,(\Lambda),$ the following
 statements hold true.
 \begin{itemize}
 \item[(a)] The matrix $C_{Q,\Theta}$ is  triangular and $\det\,(C_{Q,\Theta})=1.$
 \item[(b)] The matrix $C_{\Theta}$ is  triangular and $C_{\Theta}\,C_{Q,\Theta}=C_{Q}=C_B.$
 \item[(c)] $\det\,(C_{\Theta})=\prod_{i=1}^t\,\dim_k\End_\Lambda(\Theta(i)).$
 \end{itemize}
 \end{lem}
 \begin{dem} This is part of \cite[Lemma 2.5]{MMSZ}.
 \end{dem}
 
 \begin{teo}\label{Mepss} Let $(\Theta,\underline{Q},\leq)$ be an Ext-projective stratifying system of size $t$ in $\modu\,(\Lambda),$ $B:=\End_\Lambda(Q)^{op},$ ${}_B P(i):=\Hom_\Lambda(Q,Q(i)),$ and let  ${}_B S(i):=\mathrm{top}\,({}_B P(i)).$ Then, the following statements hold true.
 \begin{itemize}
 \item[(a)] $G_B\simeq \Coker\,(C_{\Theta})$ and $|G_B|=\prod_{i=1}^t\,\dim_k\End_\Lambda(\Theta(i)).$
 \item[(b)] For any $i\in[1,t],$ we have  $$\dim_k\End_\Lambda(\Theta(i))=[{}_B\Delta(i):{}_B S(i)]\dim_k\End({}_B S(i)).$$ 
  \item[(c)] The exponent of $G_B$ is a multiple of $\dim_k\End_\Lambda(\Theta(i)),$ for any $i\in[1,t].$
 \end{itemize}
 \end{teo}
 \begin{dem} (a) By Lemma \ref{M} (a), we have that $C_{Q,\Theta}:K_0(\add\,(Q))\to K_0(\F(\Theta))$ is an isomorphism. Therefore, from Lemma \ref{M} (b), we get 
 $G_B=\Coker\,(C_B)\simeq \Coker\,(C_{\Theta}).$ Then, by Proposition \ref{detord} and Lemma \ref{M} (c), we have
 $|G_B|=|\Coker\,(C_{\Theta})|=\det\,(C_{\Theta})=\prod_{i=1}^t\,\dim_k\End_\Lambda(\Theta(i)).$
 \
 
 (b) Let $i\in[1,t].$ Then $\dim_k\End({}_B\Delta(i))=\dim_k\Hom_B({}_B P(i),{}_B \Delta(i))=[{}_B\Delta(i):{}_B S(i)]\,\dim_k\End({}_B S(i)).$ 
 Thus, (c) follows since, by \cite[Proposition 2.12]{MMS2}, we know that $\dim_k\End({}_B\Delta(i))=\dim_k\End_\Lambda(\Theta(i)).$ 
 \
 
 (c) It follows from (a) and Proposition \ref{ddividee}, since by \cite[Lemma 2.6]{MMS2} we have
 $[C_{\Theta,B}]_{i,i}=\dim_k\Hom_\Lambda(Q(i),\Theta(i))=\dim_k\End_\Lambda(\Theta(i)).$
 \end{dem}
 \vspace{0.2cm}
 
We recall that, for any standardly stratified $k$-algebra $(\Lambda,\leq),$ we have a square matrix  playing an important role in the description of the Cartan group $G_\Lambda,$ namely,  the $\Delta$-Cartan matrix $C_\Delta \in\Mat_{n\times n}(\Z)$ where 
$[C_\Delta]_{i,j}:=[\Delta(j):S_i].$ 
 
 \begin{cor}\label{Delta-C} Let $(\Lambda,\leq)$ be a stardardly stratified $k$-algebra. Then,  the following statements hold true.
 \begin{itemize}
 \item[(a)] $G_\Lambda\simeq \Coker\,(C_\Delta)$ and $|G_\Lambda|=\prod_{i=1}^n\,\dim_k\End_\Lambda(\Delta(i)).$
 \item[(b)] If $\Lambda$ is an elementary $k$-algebra, then  $\dim_k\End_\Lambda(\Delta(i))=[\Delta(i): S_i],$  for any $i\in[1,t].$ 
 \item[(c)] The exponent of $G_\Lambda$ is a multiple of $\dim_k\End_\Lambda(\Delta(i)),$ for any $i\in[1,t].$
 \end{itemize}
 \end{cor}
 \begin{dem} We may assume that $\Lambda$ is a basic $k$-algebra. Let $\underline{Q}:=\{P_i\}_{i=i}^n$ be the complete family of pairwise non-isomorphic projective $\Lambda$-modules that we use 
 to compute the standard modules $\Delta.$ Since $(\Lambda,\leq)$ is standardly stratified, by \cite[Proposition 2.2]{MMS2} 
 we have that the triple $(\Delta,\underline{Q},\leq)$ is 
 an Ext-projective stratifying system in $\modu\,(\Lambda)$ of size $n:=rkK_0(\Lambda).$ In this case, $Q=\Lambda$ and thus $B\simeq \Lambda.$ Furthermore 
 $\dim_k\End({}_B S(i))=1,$ since $\Lambda$ is a quotient path $k$-algebra for some admissible ideal. Then, the result follows from Theorem \ref{Mepss}.
 \end{dem}
 
 \begin{cor}\label{A1}
  Let $(\Lambda, \leq)$ be an elementary and standardly stratified $k$-algebra, satisfying that $\Hom_\Lambda(P_i, \Delta(j)) =0 $ for $i< j.$ Then, for 
  $d_i:=\dim_k\,(\Delta(i)),$ we have
  $$G_\Lambda = \bigoplus^n_{i=1}\frac{\Z}{d_i\Z}.$$
 \end{cor}
 \begin{proof}  Note that  $[C_\Delta]_{i,j}=\dim_k\, \Hom(P(i), \Delta(j)) = 0$ for $i \neq j.$ Then, for any $i\in[1,n],$ the standard module $\Delta(i)$ has as composition factors only the simple $S_i.$ Moreover, $[C_\Delta]_{i,i}=d_i,$ since 
 $\dim_k\,(S_i)=1.$ Hence, by Corollary \ref{Delta-C} we get the result.
 \end{proof}
 
 \begin{rk}\label{DiagD} Let $(\Lambda, \leq)$ be a standardly stratified $k$-algebra and $n=rkK_0(\Lambda).$
 \
 
 (1) Let $i<j.$ Then $\Hom_\Lambda(P_i, \Delta(j)) =0$ implies that $\Hom_\Lambda(\Delta(i), \Delta(j)) =0.$ But the converse is not true, as can be seen in Example \ref{Ej1}.
 \
 
 (2) The following statements are equivalent
 \begin{itemize}
 \item[(a)] the matrix $C_\Delta$ is diagonal;
 \item[(b)] $\Hom_\Lambda(P_i, \Delta(j)) =0 $ for $i< j;$
 \item[(c)] $(\Lambda, \leq)$ is a weakly triangular algebra.
 \end{itemize}
We only proof that {\rm (b) $\Rightarrow$ (c).} Let $[\Delta(j):S_i] =0 $ for $i< j.$ We assert that   $[P_j:S_i]=0$ for $i< j.$  We do inverse induction on $[1,n].$
Since $P_n= \Delta_n$ the result holds if $j=n.$ For $j<n,$ we have the following short exact sequence
$$0\ra \Tr_{\oplus_{t>j}P_t}(P_j) \ra P_j \ra \Delta_j\ra 0.$$
The left side of the above sequence has, by induction, composition factors only $S_t$'s with $t > j,$ and the right hand side
has compositions factors only $S_j.$ Thus our assertion follows.
\

(3) $C_{\Delta}=\diag(1,1,\cdots,1)$ if and only if $\Lambda$ is triangular. 
\

Indeed, note that $C_{\Delta}=\diag(1,1,\cdots,1)$ 
$\Leftrightarrow$ $[\Delta(i):S_j]=\delta_{i.j}$ for all $i,j$ $\Leftrightarrow$ $\modu\,(\Lambda)=\F(\Delta).$ Then, the item (3) follows from 
\cite[Theorem 4.8]{MMSS}
 \end{rk}

 \begin{rk}  (1) One class, where the hypothesis of Corollary \ref{A1} holds, is the class of the algebras with all idempotent ideals projective. However,  there are  standardly stratified algebras, satisfying the  hypothesis of Corollary \ref{A1}, and such that not all  the idempotent ideals are projective. The 
  Example \ref{Ej2} shows that fact. 
  \
  
  (2) The Cartan group does not distinguish  iso-classes of standardly stratified algebras. Indeed, the algebras given in Example \ref{Ej2} and Example \ref{Ej3} are not isomorphic, but their Cartan group are isomorphic. Note that, in these particular examples, the Cartan matrices are different.
 \end{rk}
 
 \begin{ex}\label{Ej2}  Let $\Lambda=k\mathcal{Q}/I,$ where $\mathcal{Q}$ is the quiver 
$$
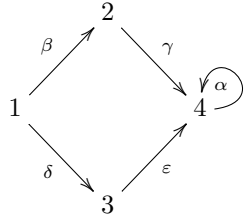
\xymatrix{& 2  \ar[dr]^{\gamma} & \\
1\ar[ur]^{\beta}\ar[dr]_\delta & & 4 \ar@(r,u)[]^{\alpha}\\
& 3  \ar[ur]_{\varepsilon} & }$$
and $I=\left\langle \gamma\beta-\varepsilon\delta,\alpha^2\right\rangle.$ The structure of the indecomposable projective $\Lambda$-modules is as follows
$$P(1)=\begin{matrix}&1&\\2&&3\\ &4&\\&4& \end{matrix},\quad P(2)=\begin{matrix}2\\4\\4 \end{matrix},\quad P(3)=\begin{matrix}3\\4\\4 \end{matrix},\quad  P(4)=\begin{matrix}4\\4 \end{matrix}.$$
Therefore $\Delta(1)=S(1),$ $\Delta(2)=S(2),$ $\Delta(3)=S(3)$ and $\Delta(4)=P(4).$ Thus, the $\Delta$-Cartan matrix is diagonal and 
${}_\Delta C=\mathrm{diag}(1,1,1,2).$ Then, by Corollary \ref{Delta-C} it follows $G_\Lambda\simeq \Z/2\Z.$ Moreover, for 
 $\overline{P}(1):=P(2)\oplus P(3)\oplus P(4),$  the idempotent ideal $\mathrm{Tr}_{\overline{P}(1)}(\Lambda)$ is not projective, since the 
 summand $\mathrm{Tr}_{\overline{P}(1)}(P(1))$ is not projective. Finally, note that 
$$C_\Lambda=\begin{pmatrix} 1 &0&0&0\\1&1&0&0\\1&0&1&0\\2&2&2&2\\
\end{pmatrix}.$$
 \end{ex}
 
 \begin{ex}\label{Ej3} Let $\Lambda=k\mathcal{Q}/I,$ where $\mathcal{Q}$ is the quiver given in Example \ref{Ej2} and $I=\left\langle\alpha^2\right\rangle.$ Note that 
 the projective modules $P(2),$ $P(3)$ and $P(4)$ are the same as in Example \ref{Ej2}. Moreover, $\rad\,P(1)=P(2)\oplus P(3)$ and thus the $\Delta$-Cartan matrix is ${}_\Delta C=\mathrm{diag}(1,1,1,2).$ Hence $G_\Lambda\simeq \Z/2\Z.$ Note that
$$C_\Lambda=\begin{pmatrix} 1 &0&0&0\\1&1&0&0\\1&0&1&0\\4&2&2&2\\
\end{pmatrix}.$$
 \end{ex}

In the following example, it is shown that any non-trivial abelian group can be realized as the Cartan group of some standardly stratified algebra.

\begin{ex}\label{Ej4} Let $(n_1,n_2,\cdots,n_k)$ be a sequence of positive integers such that $n_i\geq 2,$ for any $i\in[1,k],$ and let
$\Lambda=k\mathcal{Q}/I,$ where $\mathcal{Q}$ is the quiver
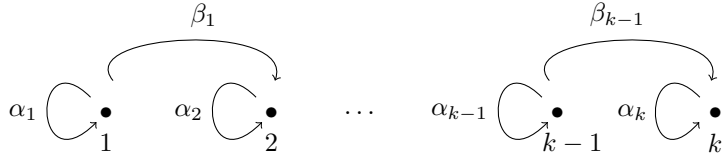
\vspace{0.2cm}

\begin{tikzpicture}
\node (c1) at (.2,0) {};
\node (c2) at (2.2,0) {};
\node (v1) at (0,-.3) {};
\node (v2) at (2.2,-.3) {};
\node (p1) at (0,-.3) {$\bullet$};
\node (n1) at (0,-.7) {$1$};
\node (p2) at (2.2,-.3) {$\bullet$};
\node (n2) at (2.2,-.7) {$2$};
\draw[->] (c1) to [out=130,in=70] (c2);
\draw[->] (p1) .. controls ($(p1)+(135:1.4)$) and ($(p1)+(225:1.4)$) .. (v1);
\draw[->] (p2) .. controls ($(p2)+(135:1.4)$) and ($(p2)+(225:1.4)$) .. (v2);
\node at ($(p2)+(-1.1,0)$) {$\alpha_2$};
\node at (1.3,1) {$\beta_1$};
\node at ($(p1)+(-1.1,0)$) {$\alpha_1$};
\node at (3.4,-.3) {$\dots$};
\node (pk1) at (6,-.3) {$\bullet$};
\node (nk1) at (6.2,-.7) {$k-1$};
\node (pk) at (8.1,-.3) {$\bullet$};
\node (nk) at (8.1,-.7) {$k$};
\node (ck1) at (6,0) {};
\node (ck) at (8,0) {};
\node (vk1) at (6,-.3) {};
\node (vk) at (8,-.3) {};
\draw[->] (ck1) to [out=130,in=70] (ck);
\draw[->] (pk1) .. controls ($(pk1)+(135:1.4)$) and ($(pk1)+(225:1.4)$) .. (vk1);
\draw[->] (pk) .. controls ($(pk)+(135:1.4)$) and ($(pk)+(225:1.4)$) .. (vk);
\node at ($(pk1)+(-1.3,0)$) {$\alpha_{k-1}$};
\node at ($(pk)+(-1.1,0)$) {$\alpha_{k}$};
\node at (6.8,1) {$\beta_{k-1}$};
\end{tikzpicture}\\
and  $I=\langle\alpha_1^{n_1},\alpha_2^{n_2},\ldots ,\alpha_{k}^{n_k}\rangle.$ The structure of the indecomposable projective 
$\Lambda$-modules $P(i)$ and $P(k)$ is \\
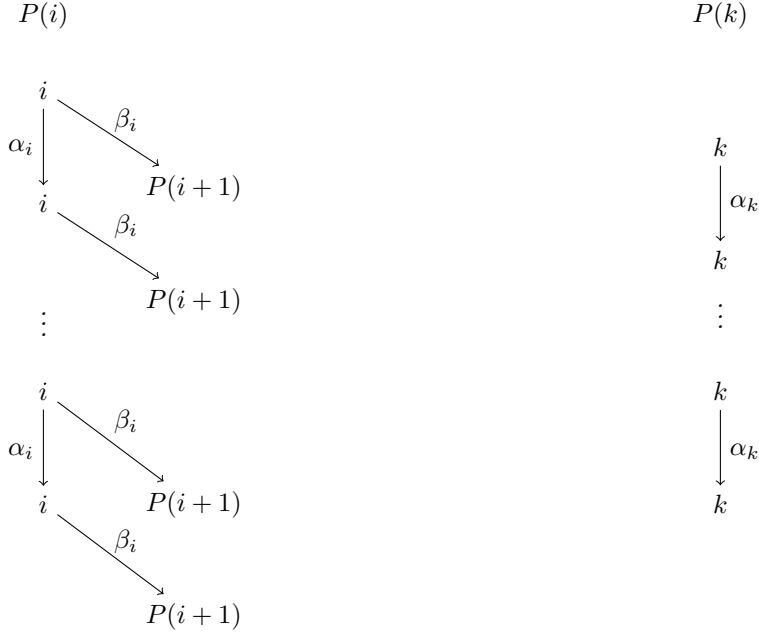
\begin{tikzpicture}
\node (p1) at (0,0) {$P(i)$};
\node (p3) at (9,0) {$P(k)$};
\node (11) at (0,-1) {$i$};
\node (12) at (0,-2.5) {$i$};
\node (r21) at (2,-2.3) {$P(i+1)$};
\node (r22) at (2,-3.8) {$P(i+1)$};
\draw[->] (11) -- (r21);
\node at (1.1,-1.4) {$\beta_i$};
\node at (1.1,-2.8) {$\beta_i$};
\draw[->] (12) -- (r22);
\node at (0,-4) {$\vdots$};
\node (13) at (0,-5) {$i$};
\node (14) at (0,-6.5) {$i$};
\node (r23) at (2,-6.5) {$P(i+1)$};
\node (r24) at (2,-8) {$P(i+1)$};
\draw[->] (13) -- (r23);
\draw[->] (14) -- (r24);
\draw[->] (11) -- (12);
\draw[->] (13) -- (14);
\coordinate [label=left:{$\alpha_i$}] (1a1) at (0,-1.75);
\coordinate [label=left:{$\alpha_i$}] (1c1) at (0,-5.8);
\node at (1.1,-5.4) {$\beta_i$};
\node at (1.1,-7) {$\beta_i$};
\node (33) at (9,-1.75) {$k$};
\node (32) at (9,-3.25) {$k$};
\node (31) at (9,-5) {$k$};
\node (30) at (9,-6.5) {$k$};
\draw[->] (33) -- (32);
\node at (9,-3.9) {$\vdots$};
\draw[->] (31) -- (30);
\coordinate [label=right:{$\alpha_k$}] (3b2) at (9,-2.5);
\coordinate [label=right:{$\alpha_k$}] (3b1) at (9,-5.8);
\end{tikzpicture}\\
Therefore, $\Delta(i)$ is uniserial and $[\Delta(i):S_j]=n_i\delta_{i,j},$ for any $i,j\in[1,k].$ Thus $C_{\Delta}=\diag(n_1,n_2,\cdots,n_k)$ and then $G_\Lambda=\oplus_{i=1}^k\,\Z/n_i\Z.$
\end{ex}

\begin{lem}\label{Cqh} Let $\Lambda$ be an artin algebra such that  $(\Lambda,\leq)$ is a standardly stratified algebra for some linear order 
on $[1,n],$ where $n:=rkK_0(\Lambda),$ and let $i\in[1,n].$ If $\End_\Lambda(\Delta(i))$ is a division ring then $[\Delta(i):S_i]=1.$
\end{lem}
\begin{dem} Let $\End_\Lambda(\Delta(i))$ be a division ring, and let $\Delta'(i):=\rad\,\Delta(i).$ Suppose that $[\Delta(i):S_i]\geq 2.$ Then, there exist submodules $M\subseteq N\subseteq\Delta'(i)$ 
such that $N/M\simeq S_i.$ Thus, there is a non zero morphism $t:P_i\to \Delta'(i)/M.$ Since $P_i$ is projective,  $t$ factors trough the canonical epimorphism 
$\pi:\Delta'(i)\to\Delta'(i)/M,$ that is, we have some $t':P_i\to \Delta'(i)$ such that $\pi t'=t.$ Let $\alpha:P_i\to \Delta(i)$ be the composition 
of $t':P_i\to \Delta'(i)$ with the inclusion $\Delta'(i)\subseteq \Delta(i).$ Note that $\alpha\neq 0,$ since $t\neq 0.$
\

Using the fact that $(\Lambda,\leq)$ is a standardly stratified algebra, for the canonical epimorphism $\pi_i:P_i\to \Delta(i),$ we have that 
$$\Hom_\Lambda(\pi_i,\Delta(i)):\End_\Lambda(\Delta(i))\to\Hom_\Lambda(P_i,\Delta(i))$$
is an isomorphism. Thus, $\exists\,\beta\in\End_\Lambda(\Delta(i))$ such that $\beta\pi_i=\alpha\neq 0$ and hence $\beta$ is an isomorphism. In 
particular $\alpha$ is an epimorphism and therefore $\rad\,\Delta(i)=\Delta(i),$ which is a contradiction. Then, we conclude that $[\Delta(i):S_i]=1.$
\end{dem}
\vspace{0.2cm}

As we will see below, for a standardly stratified algebra $\Lambda,$ the group $G_{\Lambda}$ gives a measure of how far this  algebra  is from being quasi-hereditary. 

\begin{cor}\label{A2} Let $(\Lambda, \leq)$ be a standardly stratified $k$-algebra. Then, $(\Lambda, \leq)$ is quasi-hereditary if and only if $G_\Lambda=0.$
\end{cor}
\begin{dem} Let $(\Lambda, \leq)$ be quasi-hereditary. Then, by Lemma \ref{Cqh} we get that $[\Delta(i):S_i]=1,$ for any $i\in[1,n].$ Therefore, from 
Corollary \ref{Delta-C}, it follows that $|G_\Lambda|=1.$
\

Let $G_\Lambda=0.$ Then, from Corollary \ref{Delta-C} we get that $\dim_k\End_\Lambda(\Delta(1))=1,$ for any $i\in[1,n].$
\end{dem}
\

Let  $\{P_i\}_{i=1}^n$ be a complete family of pairwise non-isomorphic indecomposable projective $\Lambda$-modules. For $i\in[1,n),$ we set  $P^{(i)}:=\oplus_{j=1}^i\,P_j,$ $\Sigma_i:=\End_\Lambda(P^{(i)})^{op},$ $\overline{P}_i:=\oplus_{j> i}\,P_j$ and 
$\Gamma_i:=\End_\Lambda(\overline{P}_i)^{op}.$

\begin{teo}\label{EqDiagM}
 Let $\Lambda=k\mathcal{Q}/I$ be a quotient path $k$-algebra, and let  $\leq$ be the natural order on $\mathcal{Q}_0:=[1,n].$
 Then, for any $i\in[1,n),$ the following statements are equivalent.
 \begin{itemize}
  \item[(a)] $C_{{}_\Lambda\Delta}$ is a diagonal matrix and \ssa is a standardly stratified algebra.
  \item[(b)] \ssa is weakly triangular. Moreover, for any $i\in[1,n),$ the pairs   $(\Sigma_i,\leq\!\!|_{[1,i]})$ and $(\Gamma_i,\leq\!\!|_{[i+1,n]})$ are standardly stratified 
  algebras satisfying that their corresponding $\Delta$-Cartan matrices  are diagonal and $\Hom_\Lambda(\overline{P}_i,P^{(i)})\in\F({}_{\Gamma_i}\Delta).$
  
  \item[(c)] 
  $\Lambda$ is isomorphic to a triangular matrix $k$-algebra  
  $\begin{pmatrix}
   R & 0\\
        M      & T
 \end{pmatrix},$ where $M$ is a $T-R$-bimodule. Moreover $(R,\leq\!\!|_{[1,i]})$ and $(T,\leq\!\!|_{[i+1,n]})$ are standardly stratified 
  algebras satisfying that their corresponding $\Delta$-Cartan matrices are diagonal and $M\in\F({}_T\Delta).$
 \end{itemize}
\end{teo} 
\begin{dem} (a) $\Rightarrow$ (b) By Remark \ref{DiagD}, we get that \ssa is weakly triangular. In particular, 
$\Hom_\Lambda(P^{(i)},\overline{P}_i)=0$ and thus 
$$\Lambda\simeq\begin{pmatrix}
\Sigma_i & 0\\ \Hom_\Lambda(\overline{P}_i,P^{(i)})& \Gamma_i
\end{pmatrix}.$$
Thus $C_{{}_\Lambda\Delta} =C_{{}_{\Sigma_i}\Delta} \oplus C_{{}_{\Gamma_i}\Delta}$ as matrices and then  
$C_{{}_{\Sigma_i}\Delta}$ and $C_{{}_{\Gamma_i}\Delta}$ are diagonal. Therefore \cite[Proposition 16]{MMS} implies (b).
\

(b) $\Rightarrow$ (c) Since \ssa is weakly triangular, we get as before the same triangular matrix decomposition of $\Lambda$ as above. Therefore, $C_{{}_\Lambda\Delta}$ is diagonal, since $C_{{}_{\Sigma_i}\Delta}$ and $C_{{}_{\Gamma_i}\Delta}$ are diagonal. Finally, 
by \cite[Proposition 16]{MMS}, we can take $R:=\Sigma_i,$ $T:=\Gamma_i$ and $M:=\Hom_\Lambda(\overline{P}_i,P^{(i)}).$
\

(c) $\Rightarrow$ (a) It follows from \cite[Proposition 16]{MMS} and the equality $C_{{}_\Lambda\Delta} =C_{{}_R\Delta} \oplus C_{{}_T\Delta}.$
\end{dem}

 In the second section, we gave some properties of presentations of finite abelian groups. These can be applied to the group $G_{\Lambda}$ for 
 a standardly stratified algebra $\Lambda,$ since the $\Delta$-Cartan matrix $C_\Delta$ is upper triangular. We turn our attention to a 
 particular case where the $\det(C_\Delta) =[C_\Delta]_{i,i}$ for some $i.$ The Example \ref{Ej2} shows that this situation occurs, more specifically, in this case $\det(C_\Delta) =[C_\Delta]_{4,4}=2.$ 

\begin{cor}\label{CEqDiagM} Let $\Lambda=k\mathcal{Q}/I$ be a quotient path $k$-algebra,  \ssa be a standardly  stratified algebra, where $\leq$ is the natural order on $\mathcal{Q}_0:=[1,n],$ and let $C_{{}_\Lambda\Delta}=\diag(d_1,d_2,\cdots,d_n).$ Then, 
$$\Lambda\simeq \begin{pmatrix}
   R & 0\\
        M      & T
 \end{pmatrix},$$
and the following statements hold true.
 \begin{itemize}
  \item[(a)] Let $d_j=1,$ for $j\in[2,n].$  Then $\Sigma_1\simeq R$ is local and $T\simeq \Gamma_1$ is triangular.
  \item[(b)] Let $d_j=1,$ for $j\in[1,n).$  Then $\Sigma_{n-1}\simeq R$ is triangular and $T\simeq \Gamma_{n-1}$ is local. Moreover, 
  ${}_T M\in\proj\,(T).$
  \item[(c)] Let $d_j=1,$ for $j\in[1,n]-\{t\},$  for some $t\in[2,n-1].$ Then $\Sigma_t\simeq R,$ with $C_{{}_R\Delta}=\diag(1,\cdots,1,d_t),$ and $T\simeq \Gamma_t$ is triangular.
 \end{itemize}
\end{cor}
\begin{proof} It follows from Remark \ref{DiagD} (3) and Theorem \ref{EqDiagM}.
 \end{proof}
 
 Let $\Lambda$ be a finite dimensional $k$-algebra. We recall that $\cP^{<\infty }(\Lambda)$ is the class of the $\Lambda$-modules 
 $M\in\modu\,(\Lambda),$ which have finite projective dimension. A very important homological measure of $\Lambda$ is the finitistic dimension 
 $\findim(\Lambda):=\sup\{\pd(M) \;:\; M\in\cP^{<\infty }(\Lambda)\}.$ On open question, till now, says that $\findim(\Lambda)$ is finite for any 
 finite dimensional $k$-algebra $\Lambda.$ 
 \
 
 Let \ssa be  a standardly  stratified algebra.  We are now interested in investigate the cases when $\F(_\Lambda\Delta) =\cP^{<\infty }(\Lambda),$ where 
 $\cP^{<\infty }(\Lambda)$ stands for the class of al $\Lambda$-modules of finite projective dimension. The preceding equality implies, by \cite[Proposition 3.17 (a)]{MMS3}, that $\findim\,(\Lambda)=\pd\,(T)\leq n-1,$ where $T$ is the characteristic tilting 
 $\Lambda$-module associated with \ssa  and $n=rkK_0(\Lambda).$

\begin{cor}\label{TP-inf} Let $\Lambda=k\mathcal{Q}/I$ be a quotient path $k$-algebra,  $\leq$ be the natural order on $\mathcal{Q}_0=[1,n]$ and let  \ssa be a standardly stratified and weakly triangular $k$-algebra. Then 
 $$\F({}_\Lambda\Delta) =\cP^{<\infty }(\Lambda).$$
\end{cor}
\begin{proof} The proof will be carried on by induction on $n= rk K_0(\Lambda).$ If $n=1,$ we get that $\Lambda$ is local and thus 
$\F(_\Lambda\Delta) =\proj\,(\Lambda)=\cP^{<\infty }(\Lambda).$
\

Let $n\geq 2.$ By Remark \ref{DiagD} (2) and Proposition \ref{EqDiagM}, we have that 
$$\Lambda \simeq \begin{pmatrix}
   \Sigma_1 & 0\\ M& \Gamma_1
 \end{pmatrix},$$
where $M:=\Hom_\Lambda(\overline{P}_1,P^{(1)})\in\F({}_{\Gamma_1}\Delta),$ $\Sigma_1$ is local and $(\Gamma_i,\leq|_{[i+1,n]})$ is standardly stratified and weakly triangular algebra. Then $\F({}_{\Sigma_1}\Delta)=\cP^{<\infty }(\Sigma_1),$  $\pd\,({}_{\Gamma_1}M)$ is finite, and by induction $\F({}_{\Gamma_1}\Delta)=\cP^{<\infty }(\Gamma_1).$ Therefore, from \cite[Theorem 19]{MMS} it follows that $\F({}_\Lambda\Delta) =\cP^{<\infty }(\Lambda).$
\end{proof}

\section{Radical square zero algebras}

In this section, we study radical square zero path $k$-algebras. Such an algebra $\Lambda$  is of the form $\Lambda=kQ/J^2,$ where $J$ is the ideal 
of $kQ$ which is generated by the set of arrows $Q_1$ of the quiver $Q.$ Let $v\in Q_0$ be a vertex of $Q.$ The idempotent in $\Lambda,$ attached to the vertex $v,$ is denoted by 
$e_v.$ Then, associated with the vertex $v,$ we have: the projective $\Lambda$-module $P(v):=\Lambda e_v$ and the simple $\Lambda$-module 
$S(v):=P(v)/\rad\,P(v).$ Given a  path $\gamma$ in $Q,$ we say that it starts at the vertex $s(\gamma)$ and ends at $t(\gamma).$

\begin{lem}\label{RS01} Let $\Lambda=kQ/J^2$ and $v\in Q_0$ be a vertex in an oriented cycle of $Q.$ Then $\pd\,S(v)=\infty.$
\end{lem}
\begin{dem} Since $\rad^2\,\Lambda=0,$ it follows that for any $w\in Q_0$ there is the exact sequence 
$0\to \bigoplus_{\alpha\in Q_1,\,s(\alpha)=w}\;S(t(\alpha))\to P(w)\to S(w)\to 0.$
\

Let $v$ be a vertex in an oriented cycle $C$ of $Q.$ If $C$  has a loop at the vertex $v,$ then by using the above exact sequence it can be constructed an
infinite minimal projective resolution of the simple $S(v)$ and thus 
$\pd\,S(v)=\infty.$
\

Suppose there is not a loop at the vertex $v.$  We can assume that $C=i_1\to i_2\to\cdots\to i_n\to i_1,$ where $v=i_1$ and all the vertices of the cycle are different to each other. By using the above exact sequence, in the vertices of $C,$ it can be constructed the following minimal long exact sequence
$$0\to S({i_1})\oplus K\to P(i_n)\oplus Q(i_n)\to\cdots \to P(i_2)\oplus Q(i_2)\to P(i_1)\to S(i_1)\to 0,$$ 
where all the $Q(i_j)'$s are projective. Then by using the above long exact sequence it can be constructed an
infinite minimal projective resolution of the simple $S(i_1);$ proving that  
$\pd\,S(i_1)=\infty.$
\end{dem}

\begin{lem}\label{RS02} Let $\Lambda=kQ/J^2$ and $\leq$ be a linear order on $Q_0$ such that $(\Lambda,\leq)$ is a standardly stratified $k$-algebra. 
If there is an arrow $v_1\to v_2$ in $Q_1$ such that $v_1<v_2,$ then $\Delta(v_2)=S(v_2).$
\end{lem}
\begin{dem} Let $v_1\to v_2$ be an arrow in $Q_1$ such that $v_1<v_2.$ Since $\rad^2\,\Lambda=0,$ the arrow $v_1\to v_2$ implies that 
$S(v_2)\subseteq P(v_1).$ Thus, the simple module $S(v_2)$ is contained in $U(v_1):=\Tr_{\bigoplus_{j>v_1}\,P(j)}(P(v_1))\in\F(\Delta).$ Note that 
$U(v_1)\subseteq \rad\,P(v_1)$ and $\rad\,P(v_1)$ is a semisimple module. Then, by using  that $\F(\Delta)$ is closed under direct summands, we get that  $S(v_2)\in\F(\Delta).$ Therefore $S(v_2)=\Delta(v)$ for some vertex $v\in Q_0.$ Finally, note that $v_2=v$ since $\Delta(v)$ has composition factors among the $S(j)$ with $j\leq v.$
\end{dem}

\begin{pro}\label{RS03} Let $\Lambda=kQ/J^2$ and $\leq$ be a linear order on $Q_0$ such that $(\Lambda,\leq)$ is a standardly stratified $k$-algebra. Then, there exists a linear order $\leq'$ on $Q_0$ such that $(\Lambda,\leq)$ is a weakly-triangular algebra.
\end{pro}
\begin{dem} By Proposition \ref{WT}, it is enough to show that $Q$ does not have proper oriented cycles. Suppose there is an oriented cycle 
$C=v_1\to v_2\to \cdots\to v_n\to v_1$ and all the vertices in this cycle are different to each other. By reordering the vertices of $\C,$ if it were 
necessary, we have that $v_1<v_2.$ Then, by Lemma \ref{RS01} and Lemma \ref{RS02}, it follows that $\pd\,\Delta(v_2)=\infty,$ which is a 
contradiction with \cite[Proposition 1.8 (i)]{AHLU}, since  $(\Lambda,\leq)$ is a standardly stratified $k$-algebra.
\end{dem}

\begin{teo} \label{RS04} Let $\Lambda=kQ/J^2$ and $\leq$ be a linear order on $Q_0$ such that $(\Lambda,\leq)$ is a standardly stratified $k$-algebra. If $Q$ does not have  sinks then $\findim\,\Lambda=0,$ $(\Lambda,\leq^{op})$ is weakly-triangular and $|G_\Lambda|=\prod_{v\in Q_0}\,[P(v):S(v)].$
\end{teo}
\begin{dem} Assume that $Q$ does not have  sinks. We show that 
$$(*)\quad \proj\,\Lambda=\F(\Delta)=\cP^{<\infty }(\Lambda).$$
Indeed, since $(\Lambda,\leq)$ is standardly stratified, it follows by \cite[Proposition 1.8(i)]{AHLU}  that 
$\proj\,\Lambda\subseteq\F(\Delta)\subseteq\cP^{<\infty }(\Lambda).$ Then, in order to obtain $(*),$ it is enough to show that any 
$\Lambda$-module of finite projective dimension is projective. Suppose there is some $M\in \cP^{<\infty }(\Lambda)$ which is not projective. 
Then, there is a minimal projective resolution $0\to P_m\xrightarrow{d_m} P_{m-1}\to\cdots\to P_1\to P_0\to M\to 0$ of $M,$ with 
$m\geq 1.$ Since $\rad\, P_{m-1}$ is semisimple and $\Ima\,(d_m)\subseteq \rad\, P_{m-1},$ we get that $P_m$ is semisimple. Therefore there is 
a simple projective $\Lambda$-module, contradicting the fact that $Q$ does not have sinks; and thus $(*)$ holds true. 
\

Note that the equalities in $(*)$ imply that  $\findim\,\Lambda=0,$ and moreover $P(v)=\Delta(v)$ for any vertex $v\in Q_0.$ Let us prove that 
$(\Lambda,\leq^{op})$ is a weakly-triangular $k$-algebra. Suppose $\Hom_\Lambda(P(i),P(j))\neq 0$ for $i>j.$ In particular, $[P(j):S(i)]\neq 0$ for 
$i>j,$ which is a contradiction with $P(j)=\Delta(j);$ proving that $(\Lambda,\leq^{op})$ is  weakly-triangular. Finally, by Corollary \ref{Delta-C} we 
conclude that $|G_\Lambda|=\prod_{v\in Q_0}\,[P(v):S(v)].$
\end{dem}

\begin{ex}\label{Ej5}  Let $\Lambda=k Q/J^2,$ where $Q$ is the quiver $\xymatrix{1\ar[r]^\alpha & 2 \ar@(r,u)[]^{\beta}}.$  The structure of the indecomposable projective $\Lambda$-modules is as follows
$$P(1)=\begin{matrix}1\\2\end{matrix}\quad\text{and}\quad  P(2)=\begin{matrix}2\\2 \end{matrix}.$$
Consider the linear order $2<1$ in $Q.$ Therefore $\Delta(1)=P(1)$ and  $\Delta(2)=P(2).$ Thus, $(\Lambda, 2<1)$ is standardly stratified, 
$(\Lambda, 1<2)$ is weakly-triangular, $C_\Delta=C_\Lambda$ and $|G_\Lambda|=[P(1):S(1)][P(2):S(2)]=2.$
 \end{ex}
 
 Let $\Lambda=kQ/I,$ where $I$ is an admissible ideal of $kQ.$ It can be defined a pre-order relation $\leq_Q,$ on the set of vertices $Q_0,$ as 
 follows. For $v,w$ in $Q_0,$ we set $v\leq_Q w$ if there is an oriented path $\gamma$ in $Q$ with $s(\gamma)=v$ and $t(\gamma)=w.$ Let 
 $\leq$ be a pre-order relation on $Q_0.$ We say that $\leq$ is a {\bf refinement} of $\leq_Q$ if the relation $x\leq_Q y$ implies that $x\leq y.$ For 
 each vertex $v\in Q_0,$ we denote by $\mathrm{loop}(v)$ the number of loops in $Q$  starting at the vertex $v.$ 
 
\begin{rk}\label{RKTR} Let $\Lambda=kQ/I,$ where $I$ is an admissible ideal of $kQ.$ In general, 
 the relation $\leq_Q$ is not an order on $Q_0.$ However,  by Proposition \ref{WT}, the following three conditions are equivalent: (a) there is a linear order 
 $\leq$ on $Q_0$ such that $(\Lambda,\leq)$ is a weakly-triangular $k$-algebra,  (b) $(Q_0,\leq_Q)$ is a partially ordered set, and (c)  the quiver $Q$ does not have proper oriented cycles.
 \end{rk}
 
 \begin{pro}\label{TR1}  Let $\Lambda=k Q/J^2,$ and let $\leq$ be a linear order on $Q_0:=[1,n]$ such that $(\Lambda,\leq)$ is a standardly stratified $k$-algebra and $\leq$ is a refinement of $\leq_Q.$ Then, the possible loops in $Q$ appear only in quasi-sources and 
 $[\Delta(i):S(j)]=(1+ \mathrm{loop}(i))\delta_{ij}$ for any $i,j\in Q_0.$ Moreover, for any $i\in[1,n],$ we have
 $1+\mathrm{loop}(i)=\dim_k\End_\Lambda(P(i))=\dim_k\End_\Lambda(\Delta(i)).$
 \end{pro}
 \begin{dem} Let $v$ be a vertex in $Q_0.$ In order to prove the result, we consider the following two cases.
 \
 
 Case 1:  Let $\mathrm{loop}(v)\geq 1.$ We show, firstly, that $v$ is a quasi-source. Indeed, suppose that $v$ is not a quasi-source. Then, 
 there is an arrow $w\to v$ in $Q_1$ with $w\neq v,$  and thus  $w<v$ since $\leq$ is a refinement of $\leq_Q.$ Therefore,  by Lemma \ref{RS01} and 
 Lemma \ref{RS02}, we get that $\pd\,\Delta(v)=\infty$ contradicting \cite[Proposition 1.8(i)]{AHLU}. 
 \
 
 Let $t:=\mathrm{loop}(v).$ Then we have only $t$ loops $\alpha_1,\alpha_2,\cdots,\alpha_t$ at the point $v,$ and possibly there are $r$ arrows 
 $v\to v_i$ in $Q_1$ with $v\neq v_i.$ In particular, we get that $v<v_i$ for any $i\in[1,r].$ Note that 
 $\rad^2\,\Lambda=0,$ and hence
 $\rad\,P(v)=S(v)^t\bigoplus\bigoplus_{i=1}^r S(v_r).$ Therefore $[\Delta(v):S(j)]=(1+ t)\delta_{vj}$ since $v<v_i$ 
 for any $i\in[1,r].$ Finally, by Corollary \ref{Delta-C} (b) and the fact that $\End_\Lambda(P(v))\simeq e_v\Lambda e_v,$ it follows
 $$\dim_k\End_\Lambda(\Delta(v))=[\Delta(v):S(v)]=t+1=\dim_k\End_\Lambda(P(v)).$$
Case 2:  Let $\mathrm{loop}(v)=0.$ Since there are not loops at $v,$  we have  $\End_\Lambda(P(v))\simeq e_v\Lambda e_v\simeq k.$ We 
assert that $\Delta(v)=S(v).$ If $v$ is not a source, then there is an arrow $w\to v$ with $v\neq w,$ and by Lemma \ref{RS02}, we get that 
$\Delta(v)=S(v).$
\

Assume now that $v$ is a source. Let $v\to v_i,$ for $i\in[1,r],$ be all possible arrows starting at the vertex $v.$ Since there are not loops at $v,$ we have that $v<v_i$ for any $i.$ Therefore $\rad\,P(v)=\bigoplus_{i=1}^r\,S(v_i)$ and hence $\Delta(v)=S(v).$
 \end{dem}
 
 \begin{pro}\label{TR2}  Let $\Lambda=k Q/J^2$ be such that $Q$ does not have proper oriented cycles and the possible loops in $Q$ appear only in quasi-sources. Then, for any linear order $\leq $ on $Q_0$ which is a refinement of $\leq_Q,$ the pair $(\Lambda,\leq)$ is a standardly stratified 
 $k$-algebra.
\end{pro} 
\begin{dem} Note that, by Remark \ref{RKTR}, the pair $(\leq_Q,Q_0)$ is a partially ordered set.  Let $\leq $  be a linear order  on $Q_0$ which is a refinement of $\leq_Q.$ For any $v\in Q_0,$ we assert that 
$$(*)\quad \mathrm{loop}(v)=0\quad\Rightarrow\quad \Delta(v)=S(v)\;\text{and}\; P(v)\in\F(\Delta).$$
Indeed, let $v\in Q_0$ be such that $\mathrm{loop}(v)=0.$ Denote by $C(v)$ the subquiver of $Q$ given by all the oriented paths starting at the 
vertex $v.$ Since the possible loops in $Q$ appear only in quasi-sources,  it follows  that $\mathrm{loop}(x)=0,$ for any vertex $x$ in $C(v).$ 
\

Consider the set $C^{(1)}(v)$ formed by all the sink vertices in $C(v).$ Note that $\Delta(k)=S(k)=P(k)$ for any $k\in C^{(1)}(v).$ Inductively, we 
 define the set $C^{(i+1)}(v)$ formed by all the vertices in $C(v)$ whose immediate successor belongs to $\bigcup_{j=1}^i\,C^{(j)}(v).$ Using that 
 $\Delta(k)=S(k)$ for any vertex $k$ in $\bigcup_{j=1}^i\,C^{(j)}(v),$ it can be shown that $\Delta(l)=S(l)$ and $P(l)\in\F(\Delta)$ for any vertex 
 $l\in C^{(i+1)}(v).$ Since the set $C(v)$ is finite, there is some natural number $m$ such that $v\in C^{(m)}(v) $ and thus the assertion in $(*)$ follows.
\

Let $w\in Q_0$ be such that $t:=\mathrm{loop}(w)\geq 1.$ There are $r$ possible 
arrows $w\to w_i$ with $w\neq w_i$ for any $i\in[1,r].$ Since the possible loops in $Q$ appear only in quasi-sources,   it follows that 
$\mathrm{loop}(w_i)=0$ for any $i.$ Then, by $(*),$ we have that $\Delta(w_i)=S(w_i)$ for any $i.$ Note that $P(w)$ has a $\Delta$-filtration, since 
$\Delta(w)=P(w)/U(w)$ and $U(w)=\bigoplus_{i=1}^r\,\Delta(v_i).$
\end{dem}
 
\begin{teo}\label{TR3} For $\Lambda=k Q/J^2,$ the following statements are equivalent.
\begin{itemize}
\item[(a)] $(\Lambda,\leq)$ is a standardly stratified $k$-algebra for some linear refinement $\leq$ of $\leq_Q.$
\item[(b)] $Q$ does not have proper oriented cycles and the possible loops in $Q$ appear only in quasi-sources.
\item[(c)] $(\Lambda,\leq)$ is a standardly stratified $k$-algebra for any linear refinement $\leq$ of $\leq_Q.$
\end{itemize}
Moreover, if one of the above equivalent conditions holds true, then the $\Delta$-Cartan matrix  $C_\Delta$ is diagonal. Moreover, for any $i\in Q_0$
 $$d_i:=[C_\Delta]_{ii}=1+\mathrm{loop}(i)=\dim_k\End_\Lambda(P(i))=\dim_k\End_\Lambda(\Delta(i)),$$
 and thus $G_\Lambda\simeq\bigoplus_{i=1}^n\,\Z/d_i\Z.$
\end{teo} 
\begin{dem} (a) $\Rightarrow$ (b) By Proposition \ref{RS03} and Remark \ref{RKTR} , we have that $Q$ does not have proper oriented cycles. Moreover, from Proposition \ref{TR1} we obtain that the possible loops in $Q$ appear only in quasi-sources.
\

(b) $\Rightarrow$ (c) This is Proposition \ref{TR2}. Finally, the implication (c) $\Rightarrow$ (a) is trivial.
\

Assume now that one of the above conditions hold true. Then, by Proposition \ref{TR1}, we get that the $\Delta$-Cartan matrix $C_\Delta$ is diagonal. Moreover, $d_i:=[C_\Delta]=1+\mathrm{loop}(i)=\dim_k\End_\Lambda(P(i))=\dim_k\End_\Lambda(\Delta(i)),$ for any vertex $i\in Q_0.$ Then, by 
Corollary \ref{Delta-C} (a), we conclude that $G_\Lambda\simeq\bigoplus_{i=1}^n\,\Z/d_i\Z.$
\end{dem}

\footnotesize

\vskip3mm \noindent Eduardo Marcos:\\
Instituto de Matem\'aticas y Estadistica,\\
Universidad de Sao Paulo,\\
Sao Paulo, BRASIL.

{\tt enmarcos@gmail.com}

\vskip3mm \noindent Octavio Mendoza:\\
Instituto de Matem\'aticas,\\
Universidad Nacional Aut\'onoma de M\'exico,\\
Circuito Exterior, Ciudad Universitaria,\\
M\'exico D.F. 04510, M\'EXICO.

{\tt omendoza@matem.unam.mx}

\vskip3mm \noindent Corina S\'aenz:\\
Departamento de Matem\'aticas, Facultad de Ciencias,\\
Universidad Nacional Aut\'onoma de M\'exico,\\
Circuito Exterior, Ciudad Universitaria,\\
M\'exico D.F. 04510, M\'EXICO.

{\tt corina.saenz@gmail.com}

\end{document}